\documentclass[a4paper]{amsart}

\usepackage[utf8]{inputenc}
\usepackage[T1]{fontenc}
\usepackage[english]{babel}

\usepackage{amsfonts}
\usepackage{amsmath}
\usepackage{amsrefs}
\usepackage{amssymb}
\usepackage{amstext}
\usepackage{amsthm}

\usepackage{geometry}
\usepackage{fullpage}

\usepackage[shortlabels]{enumitem}
\usepackage{float}
\usepackage{graphicx}
\usepackage{hyperref}
\usepackage{listings}
\usepackage{mathrsfs}
\usepackage{mathtools}
\usepackage{multicol}
\usepackage{shuffle}
\usepackage{wasysym}
\usepackage{parskip}
\usepackage{tikz}
\usepackage{tikz-cd}
\usetikzlibrary{arrows, backgrounds, calc, chains, decorations, patterns, positioning, shapes}

\newcommand{\area}{\mathsf{area}}
\newcommand{\dinv}{\mathsf{dinv}}

\newcommand{\touch}{\mathsf{touch}}

\newcommand{\D}{\mathsf{D}} 
\newcommand{\LD}{\mathsf{LD}} 
\newcommand{\LSQ}{\mathsf{LSQ}} 

\DeclareFontFamily{U}{bigshuffle}{}
\DeclareFontShape{U}{bigshuffle}{m}{n}{
	<5-8> s*[1.7] shuffle7
	<8->  s*[1.7] shuffle10
}{}
\DeclareSymbolFont{BigShuffle}{U}{bigshuffle}{m}{n}
\DeclareMathSymbol\bigshuffle{\mathop}{BigShuffle}{"001}
\DeclareMathSymbol\bigcshuffle{\mathop}{BigShuffle}{"002}

\renewcommand{\H}{\widetilde{H}}

\newcommand{\N}{\mathbb{N}}

\newcommand{\qbinom}[2]{\genfrac{[}{]}{0pt}{}{#1}{#2}}

\makeatletter

\pgfkeys{
	/tikz/sharp angle/.code={%
		\pgfsetarrowoptions{sharp >}{#1}%
		\pgfsetarrowoptions{sharp <}{-#1}%
	},
	/tikz/sharp > angle/.code={%
		\pgfsetarrowoptions{sharp >}{#1}%
	},
	/tikz/sharp < angle/.code={%
		\pgfsetarrowoptions{sharp <}{#1}%
	},
	/tikz/sharp protrude/.code=\csname if#1\endcsname\qrr@tikz@sharp@z@-0.05\p@\else\qrr@tikz@sharp@z@\z@\fi,
	/tikz/sharp protrude/.default=true
}

\newdimen\qrr@tikz@sharp@z@
\qrr@tikz@sharp@z@\z@
\pgfarrowsdeclare{sharp >}{sharp >}{%
	\edef\pgf@marshal{\noexpand\pgfutil@in@{and}{\pgfgetarrowoptions{sharp >}}}%
	\pgf@marshal
	\ifpgfutil@in@
	\edef\pgf@tempa{\pgfgetarrowoptions{sharp >}}
	\expandafter\qrr@tikz@sharp@parse\pgf@tempa\@qrr@tikz@sharp@parse
	\else
	\qrr@tikz@sharp@parse\pgfgetarrowoptions{sharp >}and-\pgfgetarrowoptions{sharp >}\@qrr@tikz@sharp@parse
	\fi
	\pgfmathparse{max(\pgf@tempa,\pgf@tempb,0)}%
	\let\qrr@tikz@sharp@max\pgfmathresult
	\pgfmathsetlength\pgf@xa{.5*\pgflinewidth * tan(\qrr@tikz@sharp@max)}%
	\pgfarrowsleftextend{+\pgf@xa}%
	\pgfarrowsrightextend{+\pgf@xa}%
}{%
	\edef\pgf@marshal{\noexpand\pgfutil@in@{and}{\pgfgetarrowoptions{sharp >}}}%
	\pgf@marshal
	\ifpgfutil@in@
	\edef\pgf@tempa{\pgfgetarrowoptions{sharp >}}
	\expandafter\qrr@tikz@sharp@parse\pgf@tempa\@qrr@tikz@sharp@parse
	\else
	\qrr@tikz@sharp@parse\pgfgetarrowoptions{sharp >}and-\pgfgetarrowoptions{sharp >}\@qrr@tikz@sharp@parse
	\fi
	\pgfmathsetlength\pgf@ya{.5*\pgflinewidth * tan(max(\pgf@tempa,\pgf@tempb,0))}%
	\pgfmathsetlength\pgf@xa{-.5*\pgflinewidth * tan(\pgf@tempa)}%
	\pgfmathsetlength\pgf@xb{-.5*\pgflinewidth * tan(\pgf@tempb)}%
	\advance\pgf@xa\pgf@ya
	\advance\pgf@xb\pgf@ya
	\ifdim\pgf@xa>\pgf@xb
	\pgftransformyscale{-1}%
	\pgf@xc\pgf@xb
	\pgf@xb\pgf@xa
	\pgf@xa\pgf@xc
	\fi
	\pgfpathmoveto{\pgfqpoint{\qrr@tikz@sharp@z@}{.5\pgflinewidth}}%
	\pgfpathlineto{\pgfqpoint{\pgf@xa}{.5\pgflinewidth}}%
	\pgfpathlineto{\pgfqpoint{\pgf@ya}{+0pt}}%
	\pgfpathlineto{\pgfqpoint{\pgf@xb}{-.5\pgflinewidth}}%
	\pgfpathlineto{\pgfqpoint{\qrr@tikz@sharp@z@}{-.5\pgflinewidth}}%
	\pgfusepathqfill
}
\pgfarrowsdeclare{sharp <}{sharp <}{%
	\edef\pgf@marshal{\noexpand\pgfutil@in@{and}{\pgfgetarrowoptions{sharp <}}}%
	\pgf@marshal
	\ifpgfutil@in@
	\edef\pgf@tempa{\pgfgetarrowoptions{sharp <}}
	\expandafter\qrr@tikz@sharp@parse\pgf@tempa\@qrr@tikz@sharp@parse
	\else
	\expandafter\qrr@tikz@sharp@parse\pgfgetarrowoptions{sharp <}and-\pgfgetarrowoptions{sharp <}\@qrr@tikz@sharp@parse
	\fi
	\pgfmathparse{max(\pgf@tempa,\pgf@tempb,0)}%
	\let\qrr@tikz@sharp@max\pgfmathresult
	\pgfmathsetlength\pgf@xa{.5*\pgflinewidth * tan(\qrr@tikz@sharp@max)}%
	\pgfarrowsleftextend{+\pgf@xa}%
	\pgfarrowsrightextend{+\pgf@xa}%
}{%
	\edef\pgf@marshal{\noexpand\pgfutil@in@{and}{\pgfgetarrowoptions{sharp <}}}%
	\pgf@marshal
	\ifpgfutil@in@
	\edef\pgf@tempa{\pgfgetarrowoptions{sharp <}}
	\expandafter\qrr@tikz@sharp@parse\pgf@tempa\@qrr@tikz@sharp@parse
	\else
	\expandafter\qrr@tikz@sharp@parse\pgfgetarrowoptions{sharp <}and-\pgfgetarrowoptions{sharp <}\@qrr@tikz@sharp@parse
	\fi
	\pgfmathsetlength\pgf@ya{.5*\pgflinewidth * tan(max(\pgf@tempa,\pgf@tempb,0))}%
	\pgfmathsetlength\pgf@xa{-.5*\pgflinewidth * tan(\pgf@tempa)}%
	\pgfmathsetlength\pgf@xb{-.5*\pgflinewidth * tan(\pgf@tempb)}%
	\advance\pgf@xa\pgf@ya
	\advance\pgf@xb\pgf@ya
	\ifdim\pgf@xa>\pgf@xb
	\pgftransformyscale{-1}%
	\pgf@xc\pgf@xb
	\pgf@xb\pgf@xa
	\pgf@xa\pgf@xc
	\fi
	\pgfpathmoveto{\pgfqpoint{\qrr@tikz@sharp@z@}{.5\pgflinewidth}}%
	\pgfpathlineto{\pgfqpoint{\pgf@xa}{.5\pgflinewidth}}%
	\pgfpathlineto{\pgfqpoint{\pgf@ya}{+0pt}}%
	\pgfpathlineto{\pgfqpoint{\pgf@xb}{-.5\pgflinewidth}}%
	\pgfpathlineto{\pgfqpoint{\qrr@tikz@sharp@z@}{-.5\pgflinewidth}}%
	\pgfusepathqfill
}
\def\qrr@tikz@sharp@parse#1and#2\@qrr@tikz@sharp@parse{\def\pgf@tempa{#1}\def\pgf@tempb{#2}}

\makeatother

\newcommand\multiset[2]%
{\mathchoice{\left(\kern-0.4em{\binom{#1}{#2}}\kern-0.4em\right)}
	{\bigl(\kern-0.2em{\binom{#1}{#2}}\kern-0.2em\bigr)}
	{\bigl(\kern-0.2em{\binom{#1}{#2}}\kern-0.2em\bigr)}
	{\bigl(\kern-0.2em{\binom{#1}{#2}}\kern-0.2em\bigr)}}

\let\existstemp\exists \renewcommand*{\exists}{\mathop \existstemp}
\let\foralltemp\forall \renewcommand*{\forall}{\mathop \foralltemp}

\def\quotient#1#2{\raise1ex\hbox{$#1$}\Big/\lower1ex\hbox{$#2$}}

\newcommand{\<}{\langle}
\renewcommand{\>}{\rangle}

\theoremstyle{plain}

\theoremstyle{definition}
\newtheorem{theorem}{Theorem}[section]
\newtheorem{conjecture}[theorem]{Conjecture}

\newtheorem{definition}[theorem]{Definition}

\theoremstyle{remark}

\makeatletter%
\renewenvironment{proof}[1][\proofname]{%
	\par\pushQED{\qed}\normalfont%
	\topsep6\p@\@plus6\p@\relax
	\trivlist\item[\hskip\labelsep\bfseries#1\@addpunct{.}]%
	\ignorespaces
}{%
	\qedhere 
}
\makeatother

\makeatletter%
\DeclareRobustCommand*{\bfseries}{%
	\not@math@alphabet\bfseries\mathbf
	\fontseries\bfdefault\selectfont
	\boldmath
}
\makeatother

\title{``Pushing'' our way from the valley Delta to the generalised valley Delta}

\author{Alessandro Iraci}
\address{Universit\'e Libre de Bruxelles (ULB)\\D\'epartement de Math\'ematique\\ Boulevard du Triomphe, B-1050 Bruxelles\\ Belgium}\email{airaci@ulb.be}

\author{Anna Vanden Wyngaerd}
\address{Universit\'e Libre de Bruxelles (ULB)\\D\'epartement de Math\'ematique\\ Boulevard du Triomphe, B-1050 Bruxelles\\ Belgium}\email{anvdwyng@ulb.ac.be}

\begin{document}
    \begin{abstract}
        In [Haglund, Remmel, Wilson 2018] the authors state two versions of the so called Delta conjecture, the rise version and the valley version. Of the former, they also give a more general statement in which zero labels are also allowed. In [Qiu, Wilson 2020], the corresponding generalisation of the valley version is also formulated.

        In [D'Adderio, Iraci, Vanden Wyngaerd 2020], the authors use a pushing algorithm to prove the generalised version of the shuffle theorem. An extension of that argument is used in [Iraci, Vanden Wyngaerd 2020] to formulate a valley version of the (generalised) Delta square conjecture, and to suggest a symmetric function identity later stated and proved in [D'Adderio, Romero 2020].

        In this paper, we use the pushing algorithm together with the aforementioned symmetric function identity in order to prove that the valley version of the Delta conjecture implies the valley version of the generalised Delta conjecture, which means that they are actually equivalent.

        Combining this with the results in [Iraci, Vanden Wyngaerd 2020], we prove that the valley version of the Delta conjecture also implies the corresponding generalised Delta square conjecture.
    \end{abstract}   
    \maketitle
    \section{Introduction}
    In \cite{Haglund-Remmel-Wilson-2018}, Haglund, Remmel and Wilson conjectured a combinatorial formula for $\Delta_{e_{n-k-1}}'e_n$ in terms of decorated labelled Dyck paths, which they called \emph{Delta conjecture}, after the so called delta operators $\Delta_f'$ introduced by Bergeron, Garsia, Haiman, and Tesler \cite{Bergeron-Garsia-Haiman-Tesler-Positivity-1999} for any symmetric function $f$. There are two versions of the conjecture, referred to as the \emph{rise} and the \emph{valley} version. 
    
    The case $k=0$ of the Delta conjecture is the famous \emph{shuffle theorem} which was proved by Carlsson and Mellit \cite{Carlsson-Mellit-ShuffleConj-2018}, using the \emph{compositional refinement} formulated in \cite{Haglund-Morse-Zabrocki-2012}. The shuffle theorem, thanks to the famous \emph{$n!$ conjecture}, now $n!$ theorem of Haiman \cite{Haiman-nfactorial-2001}, gives a combinatorial formula for the Frobenius characteristic of the $\mathfrak{S}_n$-module of diagonal harmonics studied by Garsia and Haiman.
    
    Recently, a compositional refinement of the rise version of the Delta conjecture was announced in \cite{DAdderio-Iraci-VandenWyngaerd-Theta-2020} and proved in \cite{DAdderio-Mellit-Compositional-Delta-2020}. These breakthroughs rely heavily on the novel \emph{Theta operators} introduced in \cite{DAdderio-Iraci-VandenWyngaerd-Theta-2020}. The valley version of the Delta conjecture remains an open problem today.
    
    The \emph{generalised Delta conjecture} is a combinatorial formula for $\Delta_{h_m}\Delta_{e_{n-k-1}}'e_n$ in terms of decorated partially labelled Dyck paths (the rise version first appeared in \cite{Haglund-Remmel-Wilson-2018} and the valley version in \cite{Qiu-Wilson-2020}). 
    
    Using the Theta operators, we conjectured a \emph{touching refinement} (where the number of times the Dyck path returns to the main diagonal is specified) of the valley version of the (generalised) Delta conjecture \cite{Iraci-VandenWyngaerd-2020}.  
    
    In this paper, we prove that the touching refinement of the valley version of the Delta conjecture implies the touching refinement of the valley version of the generalised Delta conjecture. Our proof will rely on a new symmetric function identity proved in \cite{DAdderio-Romero-Theta-Identities-2020}, which was suggested by a combinatorial argument we call the \emph{pushing algorithm} first described in \cite{DAdderio-Iraci-VandenWyngaerd-Theta-2020} for paths with no decorations and then extended in \cite{Iraci-VandenWyngaerd-2020} to paths with decorated contractible valleys. 
    
    Combining this result with the results in \cite{Iraci-VandenWyngaerd-2020}, we obtain that, if the valley version of the Delta conjecture is true, then the valley version of the generalised Delta square conjecture is also true. Thus, the main conjecture implies three other statements: the generalised version, the square version, and the generalised square version.
    


    \section{Symmetric functions}
    For all the undefined notations and the unproven identities, we refer to \cite{DAdderio-Iraci-VandenWyngaerd-TheBible-2019}*{Section~1}, where definitions, proofs and/or references can be found. 

    We denote by $\Lambda$ the graded algebra of symmetric functions with coefficients in $\mathbb{Q}(q,t)$, and by $\<\, , \>$ the \emph{Hall scalar product} on $\Lambda$, defined by declaring that the Schur functions form an orthonormal basis.
    
    The standard bases of the symmetric functions that will appear in our calculations are the monomial $\{m_\lambda\}_{\lambda}$, complete $\{h_{\lambda}\}_{\lambda}$, elementary $\{e_{\lambda}\}_{\lambda}$, power $\{p_{\lambda}\}_{\lambda}$ and Schur $\{s_{\lambda}\}_{\lambda}$ bases.
    
    For a partition $\mu \vdash n$, we denote by \[ \H_\mu \coloneqq \H_\mu[X] = \H_\mu[X; q,t] = \sum_{\lambda \vdash n} \widetilde{K}_{\lambda \mu}(q,t) s_{\lambda} \] the \emph{(modified) Macdonald polynomials}, where \[ \widetilde{K}_{\lambda \mu} \coloneqq \widetilde{K}_{\lambda \mu}(q,t) = K_{\lambda \mu}(q,1/t) t^{n(\mu)} \] are the \emph{(modified) Kostka coefficients} (see \cite{Haglund-Book-2008}*{Chapter~2} for more details). 
    
    Macdonald polynomials form a basis of the ring of symmetric functions $\Lambda$. This is a modification of the basis introduced by Macdonald \cite{Macdonald-Book-1995}.
    
    If we identify the partition $\mu$ with its Ferrer diagram, i.e. with the collection of cells $\{(i,j)\mid 1\leq i\leq \mu_i, 1\leq j\leq \ell(\mu)\}$, then for each cell $c\in \mu$ we refer to the \emph{arm}, \emph{leg}, \emph{co-arm} and \emph{co-leg} (denoted respectively as $a_\mu(c), l_\mu(c), a_\mu(c)', l_\mu(c)'$) as the number of cells in $\mu$ that are strictly to the right, above, to the left and below $c$ in $\mu$, respectively.
    
    Let $M \coloneqq (1-q)(1-t)$. For every partition $\mu$, we define the following constants:
    
    \[
        B_{\mu} \coloneqq B_{\mu}(q,t) = \sum_{c \in \mu} q^{a_{\mu}'(c)} t^{l_{\mu}'(c)}, \qquad \qquad
        \Pi_{\mu} \coloneqq \Pi_{\mu}(q,t) = \prod_{c \in \mu / (1,1)} (1-q^{a_{\mu}'(c)} t^{l_{\mu}'(c)}). \]
    
    We will make extensive use of the \emph{plethystic notation} (cf. \cite{Haglund-Book-2008}*{Chapter~1}).
    
    We need to introduce several linear operators on $\Lambda$.

    \begin{definition}[\protect{\cite{Bergeron-Garsia-ScienceFiction-1999}*{[3.11]}}]
        \label{def:nabla}
        We define the linear operator $\nabla \colon \Lambda \rightarrow \Lambda$ on the eigenbasis of Macdonald polynomials as \[ \nabla \H_\mu = e_{\lvert \mu \rvert}[B_\mu] \H_\mu. \]
    \end{definition}
    
    \begin{definition}
        \label{def:pi}
        We define the linear operator $\mathbf{\Pi} \colon \Lambda \rightarrow \Lambda$ on the eigenbasis of Macdonald polynomials as \[ \mathbf{\Pi} \H_\mu = \Pi_\mu \H_\mu \] where we conventionally set $\Pi_{\varnothing} \coloneqq 1$.
    \end{definition}
    
    \begin{definition}
        \label{def:delta}
        For $f \in \Lambda$, we define the linear operators $\Delta_f, \Delta'_f \colon \Lambda \rightarrow \Lambda$ on the eigenbasis of Macdonald polynomials as \[ \Delta_f \H_\mu = f[B_\mu] \H_\mu, \qquad \qquad \Delta'_f \H_\mu = f[B_\mu-1] \H_\mu. \]
    \end{definition}
    
    Observe that on the vector space of symmetric functions homogeneous of degree $n$, denoted by $\Lambda^{(n)}$, the operator $\nabla$ equals $\Delta_{e_n}$. 
    
    We also introduce the Theta operators, first defined in \cite{DAdderio-Iraci-VandenWyngaerd-Theta-2020}
    
    \begin{definition}
        \label{def:theta}
         For any symmetric function $f\in \Lambda^{(n)}$ we introduce the following \emph{Theta operators} on $\Lambda$: for every $F \in \Lambda^{(m)}$ we set
        \begin{equation*}
            \Theta_f F  \coloneqq 
            \left\{\begin{array}{ll}
                0 & \text{if } n \geq 1 \text{ and } m=0 \\
                f \cdot F & \text{if } n=0 \text{ and } m=0 \\
                \mathbf{\Pi} f^* \mathbf{\Pi}^{-1} F & \text{otherwise}
            \end{array}
            \right. ,
        \end{equation*}
    and we extend by linearly the definition to any $f,F\in \Lambda$.
    \end{definition}
    
    It is clear that $\Theta_f$ is linear, and moreover, if $f$ is homogenous of degree $k$, then so is $\Theta_f$, i.e. \[\Theta_f \Lambda^{(n)} \subseteq \Lambda^{(n+k)} \qquad \text{ for } f \in \Lambda^{(k)}. \]
    
    %
    %
    
    It is convenient to introduce the so called $q$-notation. In general, a $q$-analogue of an expression is a generalisation involving a parameter $q$ that reduces to the original one for $q \rightarrow 1$.
    
    \begin{definition}
        For a natural number $n \in \mathbb{N}$, we define its $q$-analogue as \[ [n]_q \coloneqq \frac{1-q^n}{1-q} = 1 + q + q^2 + \dots + q^{n-1}. \]
    \end{definition}
    
    Given this definition, one can define the $q$-factorial and the $q$-binomial as follows.
    
    \begin{definition}
        We define \[ [n]_q! \coloneqq \prod_{k=1}^{n} [k]_q \quad \text{and} \quad \qbinom{n}{k}_q \coloneqq \frac{[n]_q!}{[k]_q![n-k]_q!} \]
    \end{definition}
    
    \begin{definition}
        For $x$ any variable and $n \in \N \cup \{ \infty \}$, we define the \emph{$q$-Pochhammer symbol} as \[ (x;q)_n \coloneqq \prod_{k=0}^{n-1} (1-xq^k) = (1-x) (1-xq) (1-xq^2) \cdots (1-xq^{n-1}). \]
    \end{definition}
    
    We can now introduce yet another family of symmetric functions.
    
    \begin{definition}
        \label{def:Enk}
        For $0 \leq k \leq n$, we define the symmetric function $E_{n,k}$ by the expansion \[ e_n \left[ X \frac{1-z}{1-q} \right] = \sum_{k=0}^n \frac{(z;q)_k}{(q;q)_k} E_{n,k}. \]
    \end{definition}
    
    Notice that $E_{n,0} = \delta_{n,0}$. Setting $z=q^j$ we get \[ e_n \left[ X \frac{1-q^j}{1-q} \right] = \sum_{k=0}^n \frac{(q^j;q)_k}{(q;q)_k} E_{n,k} = \sum_{k=0}^n \qbinom{k+j-1}{k}_q E_{n,k} \] and in particular, for $j=1$, we get \[ e_n = E_{n,0} + E_{n,1} + E_{n,2} + \cdots + E_{n,n}, \] so these symmetric functions split $e_n$, in some sense.

    The Theta operators will be useful to restate the Delta conjectures in a new fashion, thanks to the following results.
    
    \begin{theorem}[\protect{\cite{DAdderio-Iraci-VandenWyngaerd-Theta-2020}*{Theorem~3.1} }]
        \label{thm:theta-en}
        \[ \Theta_{e_k} \nabla e_{n-k} = \Delta'_{e_{n-k-1}} e_n \]
    \end{theorem}
    
    The key symmetric function identity on which our proof relies is the following. We first formulated this identity by studying the combinatorics. Its proof is due to D'Adderio and Romero.
    
    \begin{theorem}[\cite{DAdderio-Romero-Theta-Identities-2020}*{Corollary~9.2}]
        \label{thm:sf-identity}
        Given $m,n,k,r \in \N$, we have
        \begin{align*}
            h_m^\perp \Theta_{e_k} \nabla E_{n-k,r} = \sum_{p=0}^m t^{m-p} \sum_{i=0}^p q^{\binom{i}{2}} \qbinom{r-p+i}{i}_q \qbinom{r}{p-i}_q \Delta_{h_{m-p}} \Theta_{e_{k-i}} \nabla E_{n-m-(k-i), r-p+i}.
        \end{align*}
    \end{theorem}
    
    We applied the change of variables $j \mapsto m, m \mapsto k, p \mapsto n-k-r, k \mapsto r, s \mapsto p, r \mapsto p-i$ in order to make it easier to interpret combinatorially and more consistent with the notation used in other papers.
    \section{Combinatorial definitions}
    \begin{definition}
        A \emph{Dyck path} of size $n$ is a lattice paths going from $(0,0)$ to $(n,n)$ consisting of east or north unit steps, always ending with an east step and staying above the line $x=y$, called the \emph{main diagonal}. The set of such paths is denoted by $\D(n)$. 
    \end{definition}

    \begin{definition}
        Let $\pi$ be a Dyck path of size $n$. We define its \emph{area word} to be the sequence of integers $a(\pi) = (a_1(\pi), a_2(\pi), \cdots, a_n(\pi))$ such that the $i$-th vertical step of the path starts from the diagonal $y=x+a_i(\pi)$. 
    \end{definition}
    For example the path in Figure~\ref{fig:labelled-dyck-path} has area word $01101211$.
    
    \begin{definition}
        A \emph{partial labelling} of a square path $\pi$ of size $n$ is an element $w \in \mathbb N^n$ such that
        \begin{itemize}
            \item if $a_i(\pi) > a_{i-1}(\pi)$, then $w_i > w_{i-1}$,
            \item $w_1 > 0$.
        \end{itemize}
        i.e. if we label the $i$-th vertical step of $\pi$ with $w_i$, then the labels appearing in each column of $\pi$ are strictly increasing from bottom to top, with the additional restriction that the first label cannot be a $0$.
        
        We omit the word \emph{partial} if the labelling is composed of strictly positive labels only.
    \end{definition}
    
     \begin{definition}
         A \emph{(partially) labelled Dyck path} is a pair $(\pi, w)$ where $\pi$ is a  Dyck path and $w$ is a (partial) labelling of $\pi$. We denote by $\LD(m,n)$ the set of labelled Dyck path of size $m+n$ with exactly $n$ positive labels, and thus exactly $m$ labels equal to $0$.
    \end{definition}

    
    Now we want to extend our sets introducing some decorations.
    
    \begin{definition}
        \label{def:valley}
        The \emph{contractible valleys} of a labelled square path $\pi$ are the indices $2 \leq i \leq n$ such that one either $a_i(\pi) < a_{i-1}(\pi)$, or $a_i(\pi) = a_{i-1}(\pi)$ and $w_i > w_{i-1}$.
        
        We define \[ v(\pi, w) \coloneqq \{1 \leq i \leq n \mid i \text{ is a contractible valley} \}, \] corresponding to the set of vertical steps that are directly preceded by a horizontal step and, if we were to remove that horizontal step and move it after the vertical step, we would still get a square path with a valid labelling.
    \end{definition}

    
    \begin{definition}
        A \emph{valley-decorated (partially) labelled Dyck path} is a triple $(\pi, w, dv)$ where $(\pi, w)$ is a (partially) labelled Dyck path and $dv \subseteq v(\pi, w)$.
    \end{definition}
    
    We denote by $\LD(m,n)^{\bullet k}$  the set of partially labelled valley-decorated Dyck paths of size $m+n$ with $n$ positive labels and $k$ decorated contractible valleys. 
    
    Finally, we sometimes omit writing $m$ or $k$ when they are equal to $0$. Notice that, because of the restrictions we have on the labelling and the decorations, the only path with $n=0$ is the empty path, for which also $m=0$ and $k=0$.
    
    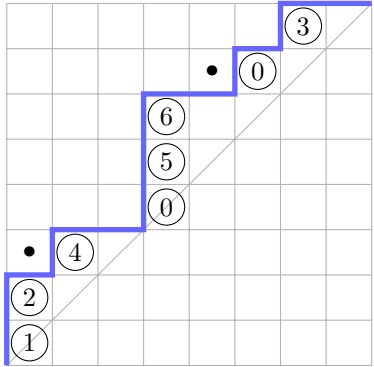
\begin{figure}[!ht]
        \centering
        \begin{tikzpicture}[scale=.6]
            \draw[step=1.0, gray!60, thin] (0,0) grid (8,8);
        
            \begin{scope}
                \clip (0,0) rectangle (8,8);
                \draw[gray!60, thin] (0,0) -- (8,8);
            \end{scope}
        
            \draw[blue!60, line width=2pt] (0,0) -- (0,1) -- (0,2) -- (1,2) -- (1,3) -- (2,3) -- (3,3) -- (3,4) -- (3,5) -- (3,6) -- (4,6) -- (5,6) -- (5,7) -- (6,7) -- (6,8) -- (7,8) -- (8,8);
        
            \draw (0.5,0.5) circle (0.4cm) node {$1$};
            \draw (0.5,1.5) circle (0.4cm) node {$2$};
            \draw (1.5,2.5) circle (0.4cm) node {$4$};
            \draw (3.5,3.5) circle (0.4cm) node {$0$};
            \draw (3.5,4.5) circle (0.4cm) node {$5$};
            \draw (3.5,5.5) circle (0.4cm) node {$6$};
            \draw (5.5,6.5) circle (0.4cm) node {$0$};
            \draw (6.5,7.5) circle (0.4cm) node {$3$};
            \draw (0.5,2.5) node {$\bullet$};
            \draw (4.5,6.5) node {$\bullet$};
        \end{tikzpicture}
        \caption{Example of an element in $\LD(2,6)^{\bullet 2}$.}
        \label{fig:labelled-dyck-path}
    \end{figure}

    \begin{definition}
        Let $w$ be a labelling of Dyck path of size $n$. We define $x^w \coloneqq \prod_{i=1}^{n} x_{w_i} \rvert_{x_0 = 1}$. For $P \coloneqq (\pi,w,dv)\in \LD(m,n)^{\bullet k}$ we define $x^P\coloneqq x^w$. 
    \end{definition}
    
    The fact that we set $x_0 = 1$ explains the use of the expression \emph{partially labelled}, as the labels equal to $0$ do not contribute to the monomial.
    
    \begin{definition}\label{def:touch}
        Let $P \coloneqq (\pi,w,dv) \in \LD(m,n)^{\bullet k}$. Define a \emph{touching point} of $P$ to be the a point on the main diagonal that is the starting point of a non-decorated vertical step of $P$, labelled with a positive label. The \emph{touch} of a path $P$, denoted $\touch(P)$, is the number of  touching points of $P$. 
    \end{definition}
    For example the path in Figure~\ref{fig:labelled-dyck-path} has touch $1$. 
                           
    We define two statistics on this set. 
    
    \begin{definition}\label{def:area}
        For $(\pi, w, dv) \in \LD(m,n)^{\bullet k}$ we define 
        \[ \area(\pi, w, dv) \coloneqq \sum_{i =1}^{m+n} a_i(\pi), \]
        i.e. the number of whole squares between the path and the main diagonal.
    \end{definition}
    
    For example, the path in Figure~\ref{fig:labelled-dyck-path} has area $7$. Notice that the area does not depend on the labelling or decorations.
    
    \begin{definition}\label{def:dinv}
        Let $(\pi, w, dv) \in \LD(m,n)^{\bullet k}$. For $1 \leq i < j \leq n$, the pair $(i,j)$ is a \emph{diagonal inversion} if
        \begin{itemize}
            \item either $a_i(\pi) = a_j(\pi)$ and $w_i < w_j$ (\emph{primary inversion}),
            \item or $a_i(\pi) = a_j(\pi) + 1$ and $w_i > w_j$ (\emph{secondary inversion}),
        \end{itemize}
        where $w_i$ denotes the $i$-th letter of $w$, i.e. the label of the vertical step in the $i$-th row. Then we define 
        \begin{align*}
            \dinv(\pi, w, dv) \coloneqq \# \{ 1 \leq i < j \leq n \mid (i,j) \text{ diagonal inversion } \land i \not \in dv \} - \# dv.
        \end{align*}
    \end{definition}
    
    For example, the path in Figure has $3$ primary inversions ($(2,3),(2,5)$ and $(2,8)$), $3$ secondary inversions ($(2,4),(6,7)$ and $(6,8)$) and $2$ decorated valleys. So its dinv equals $3+3-2=4$.
    
    It is easy to check that if $j \in dv$ then either there exists some diagonal inversion $(i,j)$ and so the dinv is always non-negative (see \cite{Iraci-VandenWyngaerd-2020}*{Proposition~1}).

    \section{The valley generalised Delta conjecture}
    Now we have all the tools to state the conjectural formulas that are the object of this paper. The following conjecture was first stated in \cite{Haglund-Remmel-Wilson-2018}.

    \begin{conjecture}[Delta conjecture, valley version]
        \label{conj: Delta conjecture, valley version}
        For $n,k\in \N$ with $k<n$
        \begin{align*}
            \Delta'_{e_{n-k-1}} e_n = \sum_{P\in \LD(n)^{\bullet k}}q^{\dinv(P)}t^{\area(P)}x^{P}.
        \end{align*}	
    \end{conjecture}
    
    In \cite{Qiu-Wilson-2020}, the authors proposed the following formula, containing the previous one as a special case ($m=0$).
    \begin{conjecture}[Generalised Delta conjecture, valley version]
        \label{conj: Generalised Delta conjecture, valley version}
        For $m,n,k\in \N$ with $k<n$
        \begin{align*}
            \Delta_{h_m}\Delta'_{e_{n-k-1}} e_n = \sum_{P\in \LD(m,n)^{\bullet k}}q^{\dinv(P)}t^{\area(P)}x^{P}.
        \end{align*}	
    \end{conjecture}
    
    Recall that, using Theorem~\ref{thm:theta-en}, the symmetric function can be reformulated using the Theta operators as follows: 
    \begin{align*}
        \Delta_{h_m}\Delta'_{e_{n-k-1}} e_n = \Delta_{h_m}\Theta_{e_k}\nabla e_{n-k}. 
    \end{align*}
    We have the following refinements, first stated in \cite{Iraci-VandenWyngaerd-2020}. 
    \begin{conjecture}[Touching Delta conjecture, valley version]\label{conj:hypothesis}
        For $n,k,r\in \N$ with $k<n$
        \begin{align*}
            \Theta_{e_k} \nabla E_{n-k,r} =  \sum_{\substack{P\in \LD(n)^{\bullet k} \\ \touch(P)=r}}q^{\dinv(P)}t^{\area(P)}x^{P}.
        \end{align*}
    \end{conjecture}
    \begin{conjecture}[Touching generalised Delta conjecture, valley version]\label{conj:thesis}
        For $m,n,k,r\in \N$ with $k<n$
        \begin{align*}
            \Delta_{h_m}\Theta_{e_k} \nabla E_{n-k,r} =  \sum_{\substack{P\in \LD(m,n)^{\bullet k} \\ \touch(P)=r}}q^{\dinv(P)}t^{\area(P)}x^{P}.
        \end{align*}
    \end{conjecture}
    
    The goal of this paper is to prove that Conjecture~\ref{conj:hypothesis} implies Conjecture~\ref{conj:thesis}.
    
    As a corollary, we obtain an analogous result for the corresponding square conjecture. We refer to \cite{Iraci-VandenWyngaerd-2020} for the missing definitions. In that paper, we introduced the aforementioned square analogue of the valley version of the Delta conjecture, and we showed that it is implied by Conjecture~\ref{conj:thesis}
    
    \begin{conjecture}[Modified Delta square conjecture, valley version]
        \label{conj:valley-square-2}
        \[ \Theta_{e_k} \nabla \omega(p_{n-k}) = \sum_{P \in \LSQ'(n)^{\bullet k}} q^{\dinv(P)} t^{\area(P)} x^P. \]
    \end{conjecture}
    
    This conjecture extends nicely to the $m>0$ case, and we also showed that the generalised valley Delta conjecture implies the corresponding square analogue, that is, Conjecture~\ref{conj:thesis} implies Conjecture~\ref{conj:gen-valley-square-2}.
    
    \begin{conjecture}[Modified generalised Delta square conjecture, valley version]
        \label{conj:gen-valley-square-2}
        \[ \Delta_{h_m} \Theta_{e_k} \nabla \omega(p_{n-k}) = \sum_{P \in \LSQ'(m,n)^{\bullet k}} q^{\dinv(P)} t^{\area(P)} x^P. \]
    \end{conjecture}
    
    Combining these results with the main result of this paper, we get that, if Conjecture~\ref{conj:hypothesis} holds, then Conjecture~\ref{conj:thesis}, Conjecture~\ref{conj:valley-square-2}, and Conjecture~\ref{conj:gen-valley-square-2} all hold.
    
    \section{The proof}
        Our proof comprises two steps: first, we will interpret \ref{thm:sf-identity} combinatorially. Then we will apply an induction argument on the $m$ variable of the same equation to conclude. 

        The left hand side of \ref{thm:sf-identity},  $h_m^\perp \Theta_{e_k}\nabla E_{n-k,r}$, coincides with $h_m^\perp$ applied to the left hand side of \ref{conj:hypothesis}.  Applying $h_m^\perp$ to the right hand side of \ref{conj:hypothesis} has the effect\footnotemark of selecting all the paths that have exactly $m$ maximal labels, and setting the variable of that label equal to $1$. In other words, if for a path $P$ of labelling $w$ we define $\max(P)\coloneqq \max(w)$, \ref{conj:hypothesis} implies 
        \begin{align}\label{eq:hperp-combinatorics }
        h_m^\perp \Theta_{e_k} \nabla E_{n-k,r} =\sum_{\substack{P\in \LD(n)^{\bullet k}\\ \touch(P)=r\\ P \text{ has } m \text{ maximal labels}}} q^{\dinv(P)}t^\area{(P)}x^P\rvert_{x_{\max(P)} = 1}.
        \end{align}
        \footnotetext{Indeed, by definition $\< h_m^\perp \Theta_{e_k} \nabla E_{n-k,r}, h_\mu \> = \< \Theta_{e_k} \nabla E_{n-k,r}, h_m h_\mu \>$ and the homogeneous basis is dual to the monomial basis with respect to the Hall scalar product.}
        Thus the combinatorial counterpart of \ref{thm:sf-identity} is the following. 
        \begin{theorem} For all $m,n,r,k\in \N$ we have
        \label{thm:pushing-theorem}
        \begin{align*}\sum_{\substack{P\in \LD(n)^{\bullet k}\\ \touch(P)=r\\ P \text{ has } m \text{ maximal labels}}}& q^{\dinv(P)}t^\area{(P)}x^P\rvert_{x_{\max(P)} = 1}
        \\&= 
        \sum_{p=0}^m t^{m-p} 
        \sum_{i=0}^p q^{\binom{i}{2}} \qbinom{r-p+i}{i}_q \qbinom{r}{p-i}_q 
        \sum_{\substack{P \in \LD(m-p, n-m)^{\bullet k-i} \\ \touch(P)=r-p+i }} q^{\dinv(P)} t^{\area(P)} x^P 
        \end{align*}
        \end{theorem}
        \begin{proof}
        Start from an element enumerated in the left hand side of the equation: a Dyck path $P$ of size $n$, with $k$ decorations on valleys, touch $r$ and $m$ maximal labels. 
        We apply what we call the \emph{pushing algorithm}, which comprises two operations. Note that any vertical step $v$ labelled with a maximal label must be followed by a horizontal step $h$. Let $v$ be any such step.
        \begin{enumerate}
        \item If the starting point of $v$ lies on the main diagonal, delete $v$ and $h$. If $v$ was a decorated valley, this decoration also gets deleted.
        \item If the starting point of $v$ does not lie on the main diagonal, replace $vh$ by $hv$ and change the label of $v$ (which was a maximal label) to $0$. If $v$ was a decorated valley, it remains so. See Figure~\ref{fig:push-label}. This operation yields a valid Dyck path since $v$ did not touch the main diagonal. The labelling also stays valid as $0$ is smaller than any label of $P$.    
        \end{enumerate}
    \begin{figure}[!ht]
        \begin{minipage}{.5\textwidth}
        \begin{center}
            \begin{tikzpicture}[scale = .6]
            \draw[blue!60, line width = 2pt] (0,0)|-(1,1) (2,0)-|(3,1);
            \draw (.5,.5) circle (.4cm) node {$M$};
            \draw (3.5,.5) circle (.4cm) node {$0$};
            \node at (1.5,.5) {$\rightarrow$}; 
            \end{tikzpicture}  
        \end{center}  
        \end{minipage}%
        \begin{minipage}{.5\textwidth}
        \begin{center}
            \begin{tikzpicture}[scale = .6]
            \draw[blue!60, line width = 2pt] (0,0)|-(1,1) (2,0)-|(3,1);
            \draw (.5,.5) circle (.4cm) node {$M$};
            \draw (3.5,.5) circle (.4cm) node {$0$};
            \node at (1.5,.5) {$\rightarrow$}; 
            \node at (-.5,.5) {$\bullet$};
            \node at (2.5,.5) {$\bullet$};
        \end{tikzpicture}  
        \end{center}
        \end{minipage}
        \caption{``Pushing'' a step labelled with a maximal label $M$.}\label{fig:push-label}
    \end{figure}
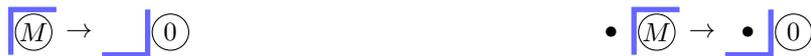
    We apply this procedure to all $m$ steps labelled with a maximal label. See Figure~\ref{fig:push-algorithm} for an example. Let $p$ be the number of vertical steps starting from the main diagonal with a maximal label. Let $i$ be the number of such steps that are decorated valleys. It follows that after applying the pushing algorithm, we obtain a path $\tilde{P}$ of size $n-p$, with $k-i$ decorations and $m-p$ zero labels. Thus, $\tilde{P}\in \LD(m-p,n-m)^{\bullet k-i}$. Since the touch does not take into account steps labelled $0$ or decorated steps starting from the main diagonal, the touch of $\tilde{P}$ is $r-(p-i)$. 
    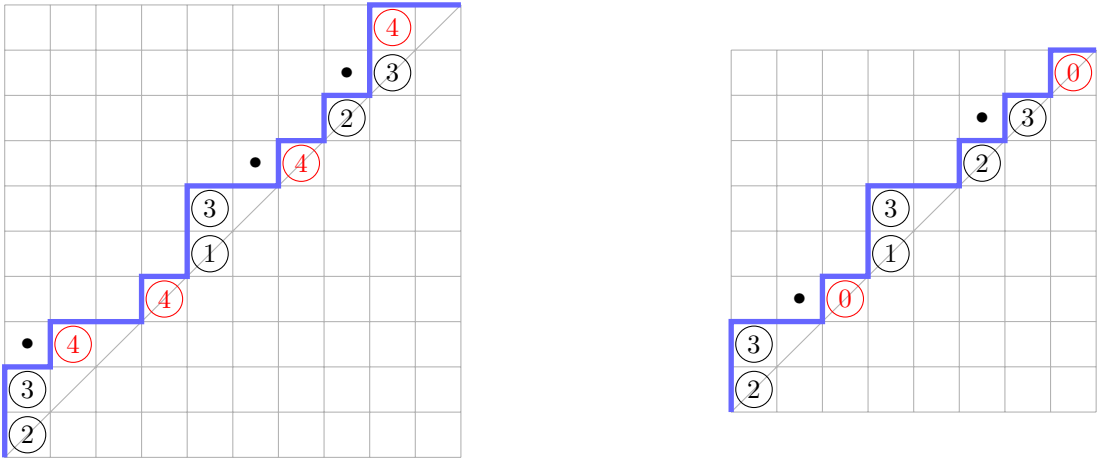
\begin{figure}[!ht]
        \begin{minipage}{.6\textwidth}
        \begin{tikzpicture}[scale = .6]
            \draw[gray!60, thin](0,0) grid (10,10) (0,0)--(10,10);      
            \draw[blue!60, line width=2pt] (0,0)|-(1,2)|-(3,3)|-(4,4)|-(6,6)|-(7,7)|-(8,8)|-(10,10);
            \draw (.5,.5) circle (.4cm) node {$2$};
            \draw (.5,1.5) circle (.4cm) node {$3$};
            \draw[red] (1.5,2.5) circle (.4cm) node {$4$};
            \draw[red] (3.5,3.5) circle (.4cm) node {$4$};
            \draw (4.5,4.5) circle (.4cm) node {$1$};
            \draw (4.5,5.5) circle (.4cm) node {$3$};
            \draw[red] (6.5,6.5) circle (.4cm) node {$4$};
            \draw (7.5,7.5) circle (.4cm) node {$2$};
            \draw (8.5,8.5) circle (.4cm) node {$3$};
            \draw[red] (8.5,9.5) circle (.4cm) node {$4$};
            \node at (.5,2.5) {$\bullet$};
            \node at (5.5,6.5) {$\bullet$};
            \node at (7.5,8.5) {$\bullet$};
        \end{tikzpicture}
        \end{minipage}%
        \begin{minipage}{.4\textwidth}
        \begin{tikzpicture}[scale = .6]
            \draw[gray!60, thin](0,0) grid (8,8) (0,0)--(8,8);      
            \draw[blue!60, line width=2pt] (0,0)|-(2,2)|-(3,3)|-(5,5)|-(6,6)|-(7,7)|-(8,8);
            \draw (.5,.5) circle (.4cm) node {$2$};
            \draw (.5,1.5) circle (.4cm) node {$3$};
            \draw[red] (2.5,2.5) circle (.4cm) node {$0$};
    
            \draw (3.5,3.5) circle (.4cm) node {$1$};
            \draw (3.5,4.5) circle (.4cm) node {$3$};
    
            \draw (5.5,5.5) circle (.4cm) node {$2$};
            \draw (6.5,6.5) circle (.4cm) node {$3$};
            \draw[red] (7.5,7.5) circle (.4cm) node {$0$};
            \node at (1.5,2.5) {$\bullet$};
            \node at (5.5,6.5) {$\bullet$};
        \end{tikzpicture}
        \end{minipage}
        \caption{The pushing algorithm.}\label{fig:push-algorithm}
    \end{figure}
    
    Clearly, performing (1) does not change the area and performing (2) reduces the area by one unit. Since we apply (2) $m-p$ times, we have 
    \begin{align*}
        \area(P)= \area(\tilde{P}) + m-p;
    \end{align*} which explains the factor $t^{m-p}$.
    
    Let us now study what happens to the dinv. Performing (2) does not alter the dinv since any primary dinv pair involving $v$ becomes a secondary dinv pair and vice versa. For (1), let us distinguish three types of steps on the main diagonal of $P$:
    \begin{enumerate}[(a)]
        \item non-decorated steps with a maximal label, of which there are $p-i$;
        \item decorated steps with a maximal label, of which there are $i$;
        \item non-decorated steps with a non-maximal label, of which there are $r-(p-i)$. 
    \end{enumerate}
    The steps of type (a) and (b) get deleted by the algorithm, so we must determine how they contribute to the dinv. The only dinv created by steps of type (a) is primary dinv with steps of type (c). So the contribution to the dinv for the steps of type (a) depends on the interlacing of these two types of steps and is $q$-counted by $\qbinom{r}{p-i}_q$. Similarly, the only dinv created by steps of type (b) is primary dinv with steps of type (c). The contractibility of the decorated valleys implies that 
    \begin{itemize}
        \item there must be a step of type (c) before the first occurrence of a step of type (b);
        \item between two steps of type (b), there must be a step of type (c);
    \end{itemize}
    However, there may be a step of type (b) after all the steps of type (c). Thus, the contribution to the dinv for the steps of type (b) is $q$-counted $q^{\binom{i}{2}}\qbinom{r-(p-i)}{i}_q$. 
    
    Taking the sum over all the possible $p$'s and $i$'s, we get the announced formula. 
    \end{proof}

    \begin{theorem}[Conditional generalised Delta conjecture, valley version] \hfill \\
        If for $n,k,r \in \N$ the identity \[ \Theta_{e_k} \nabla E_{n-k, r} = \sum_{\substack{P \in \LD(n)^{\bullet k}\\ \touch(P)= r}} q^{\dinv(P)} t^{\area(P)} x^P \] holds, then for $m,n,k,r \in \N$, the identity \[ \Delta_{h_m} \Theta_{e_k} \nabla E_{n-k, r} = \sum_{\substack{P \in \LD(m, n)^{\bullet k}\\ \touch(P)= r}} q^{\dinv(P)} t^{\area(P)} x^P \] also holds.
    \end{theorem}
    
    \begin{proof}
        We proceed by induction on $m$. For $m=0$, the statement is exactly the valley version of the Delta conjecture, which we are assuming to hold.
    
        For $m > 0$, by Theorem~\ref{thm:sf-identity} we have
        \begin{align*}
            h_m^\perp \Theta_{e_k} \nabla E_{n-k,r} = \sum_{p=0}^m t^{m-p} \sum_{i=0}^p q^{\binom{i}{2}} \qbinom{r-p+i}{i}_q \qbinom{r}{p-i}_q \Delta_{h_{m-p}} \Theta_{e_{k-i}} \nabla E_{n-m-(k-i), r-p+i}.
        \end{align*}
        
    
        By Theorem~\ref{thm:pushing-theorem}, we can rewrite the statement of Theorem~\ref{thm:sf-identity} as 
    
        \begin{align*}
            \sum_{p=0}^m & t^{m-p} \sum_{i=0}^p q^{\binom{i}{2}} \qbinom{r-p+i}{i}_q \qbinom{r}{p-i}_q \sum_{\substack{P \in \LD(m-p, n-m )^{\bullet k-i}\\ \touch(P)= r-p+i}} q^{\dinv(P)} t^{\area(P)} x^P \\ & = \sum_{p=0}^m t^{m-p} \sum_{i=0}^p q^{\binom{i}{2}} \qbinom{r-p+i}{i}_q \qbinom{r}{p-i}_q \Delta_{h_{m-p}} \Theta_{e_{k-i}} \nabla E_{n-m-(k-i), r-p+i}.
        \end{align*}
    
        By induction hypothesis, whenever $p > 0$ we have \[ \Delta_{h_{m-p}} \Theta_{e_{k-i}} \nabla E_{n-m-(k-i), r-p+i} = \sum_{\substack{P \in \LD(m-p, n-m )^{\bullet k-i}\\ \touch(P) =  r-p+i}} q^{\dinv(P)} t^{\area(P)} x^P \]
    
        so all the terms of the sum except the one when $p=0$ cancel out, and we are left with \[ t^m \sum_{\substack{P \in \LD(m, n-m)^{\bullet k}\\ \touch(P)= r}} q^{\dinv(P)} t^{\area(P)} x^P = t^m \Delta_{h_m} \Theta_{e_k} \nabla E_{n-m-k, r} \] which is, up to the substitution $n \mapsto n+m$ and a division by $t^m$, exactly what we wanted to show.
        \end{proof}
    \section{Acknowledgements}The authors would like to thank Michele D'Adderio for the many interesting discussions on the topic.
    \bibliography{bibliography}
    \bibliographystyle{amsalpha}
\end{document}